\documentclass[11pt, twoside]{article}
\usepackage{amsfonts}
\usepackage{amsmath,amssymb}
\usepackage{amsthm}
\usepackage[mathscr]{euscript}
\usepackage{mathrsfs}
\usepackage{indentfirst}
\usepackage{color}
\usepackage{enumitem}
\usepackage{accents}

\newcommand{\lbl}[1]{\label{#1}}

\allowdisplaybreaks
\newtheorem{theo}{Theorem}[section]

\newcommand{\be}{\begin{equation}}
\newcommand{\ee}{\end{equation}}
\newcommand\bes{\begin{eqnarray}} \newcommand\ees{\end{eqnarray}}
\newcommand{\bess}{\begin{eqnarray*}}
\newcommand{\eess}{\end{eqnarray*}}

\newcommand\ep{\varepsilon}

\newcommand\dd{\displaystyle}
\newcommand\vp{\varphi}

\newcommand\lm{\lambda}

\newcommand\yy{\infty}

\newcommand\R{\mathbb{R}}

\newcommand\ud{\underline}

\newcommand\oo{\Omega}

\newcommand\qq{\eqref}
\newcommand\sst{\scriptscriptstyle}
\begin{document}\thispagestyle{empty}
\begin{center}{\Large\bf Note on a vector-host epidemic model with spatial structure}\footnote{This work was
supported by NSFC Grant 12171120.}\\[4mm]
 {\Large Mingxin Wang\footnote{{\sl E-mail}: mxwang@hpu.edu.cn }}\\[0.5mm]
 {\small School of Mathematics and Information Science, Henan Polytechnic University, Jiaozuo  454003, China}
\end{center}

\begin{quote}
\noindent{\bf Abstract.}  Magal,  Webb and Wu [Nonlinearity 31, 5589-5614 (2018)] studied the model describing outbreak of Zika in Rio De Janerio, and provided a complete analysis of dynamical properties for the
solutions. In this note we first use a very simple approach to prove their results, and then investigate the modified version of the model concerned in their paper, with Neumann boundary condition replaced by Dirichlet boundary condition.


\end{quote}

\pagestyle{myheadings}
\section{Introduction}{\setlength\arraycolsep{2pt}
\markboth{\rm$~$ \hfill Note on a vector-host epidemic model\hfill $~$}{\rm$~$ \hfill M.X. Wang\hfill $~$}

Vector-borne diseases such as chikungunya, dengue, malaria, West Nile virus, yellow fever and Zika virus pose a major threat to global public health \cite{PAHO1, PAHO2, WHO1, WHO2}. In recent years, many authors
have proposed reaction-diffusion models to study the transmission of diseases in spatial settings. Among them, Fitzgibbon et al. \cite{FMW} proposed and studied a spatially correlated bounded vector-host epidemic model that qualitatively described the 2015-2016 Zika epidemic in Rio de Janeiro. The global dynamics of the spatial model were further established based on the basic reproduction number. Suppose that individuals are living in a bounded smooth domain $\oo\subset \R^n$, and $\nu$ is the outward normal vector of $\partial\Omega$. Let $H_i(x, t), V_u(x, t)$ and $V_i(x, t)$ be the densities of infected hosts, uninfected vectors, and infected vectors at position $x$ and time $t$, respectively. Then the model proposed in \cite{FMW} to
study the outbreak of Zika in Rio De Janerio is the following reaction-diffusion system\vspace{-1mm}
 \begin{equation}\begin{cases}
\partial_t H_i-\nabla\cdot d_1(x)\nabla H_i=-\rho(x)H_i+\sigma_1(x) H_u(x)V_i,\, &x\in\oo,\; t>0,\\
\partial_t V_u-\nabla\cdot d_2(x)\nabla V_u=-\sigma_2(x)V_uH_i +\beta(x)(V_u+V_i)-\mu(x)(V_u+V_i)V_u, &x\in\oo,\; t>0,\\
\partial_t V_i-\nabla\cdot d_2(x)\nabla V_i=\sigma_2(x)V_uH_i-\mu(x)(V_u+V_i)V_i,\, &x\in\oo,\; t>0,\\
\partial_\nu H_i=\partial_\nu V_u=\partial_\nu V_i=0,\;\;&x\in\partial\oo,\; t>0,\\
(H_i(\cdot, 0), V_u(\cdot, 0), V_i(\cdot, 0))=(H_{i0}, V_{u0}, V_{i0})\in C(\bar\oo;\,\mathbb{R}^3_+),
	\end{cases}\label{1.1}\vspace{-1mm}
\end{equation}
where $d_1, d_2\in C^{1+\alpha}(\bar\oo)$ and $\rho, \beta, \sigma_1, \sigma_2, \mu\in C^\alpha(\bar\oo)$ are strictly positive, and $H_u\in C^\alpha(\bar\oo)$ is nonnegative and non-trivial. The flux of new infected humans is given by $\sigma_1(x) H_u(x)V_i(x,t)$ in which $H_u(x)$ is the density of susceptible population depending on the spatial location $x$. The initial functions satisfy $\partial_\nu H_{i0}=\partial_\nu V_{u0}=\partial_\nu V_{i0}=0$ on $\partial\oo$.

When establishing the model \qq{1.1}, the main assumption is that the susceptible human population is (almost) not affected by the epidemic during a relatively short period of time and therefore the flux of new infected is (almost) constant. Such a functional response mainly permits to take care of realistic density of population distributed in space.

The model \qq{1.1} has been thoroughly studied by Magal, Webb and Wu \cite{MWW18,MWW19}. In \cite{MWW18}, they found  the basic reproduction number $\mathcal{R}_0$ and show that $\mathcal{R}_0$ is a threshold parameter: if $\mathcal{R}_0\le 1$ the disease free equilibrium is globally stable; if $\mathcal{R}_0>1$ the model has a unique globally stable positive equilibrium. In \cite{MWW19}, it was shown that the basic reproduction number $\mathcal{R}_0$ can be defined as the spectral radius of a product of a local basic reproduction number $\mathcal{R}$ and strongly positive compact linear operators
with spectral radii one.

In this note we first use a very simple approach to prove the main results of \cite{MWW18}, and then study the other version of \qq{1.1} coupled with homogeneous Dirichlet boundary condition
\begin{equation}
	\begin{cases}
\partial_t H_i-\nabla\cdot d_1(x)\nabla H_i=-\rho(x)H_i+\sigma_1(x) H_u(x)V_i,\, &x\in\oo,\; t>0,\\
\partial_t V_u\!-\!\nabla\cdot d_2(x)\nabla V_u=-\sigma_2(x)V_uH_i \!+\!\beta(x)(V_u\!+\!V_i)\!-\!\mu(x)(V_u\!+\!V_i)V_u, &x\in\oo,\; t>0,\\
\partial_t V_i-\nabla\cdot d_2(x)\nabla V_i=\sigma_2(x)V_uH_i-\mu(x)(V_u+V_i)V_i,\, &x\in\oo,\; t>0,\\
 H_i=V_u=V_i=0,&x\in\partial\oo,\; t>0,\\
H_i(x, 0)=H_{i0}(x), V_u(x, 0)=V_{u0}(x), V_i(x, 0)=V_{i0}(x),
	\end{cases}
	\label{1.2}
 \end{equation}
where the initial data $H_{i0}, V_{u0}, V_{i0}\in C(\bar\oo)$, and are positive in $\oo$ and satisfy $H_{i0}=V_{u0}=V_{i0}=0$ on $\partial\oo$, all parameters are same as in \qq{1.1}. It is easy to see that the solutions $(H_i, V_u, V_i)$ of \qq{1.1} and \qq{1.2} are unique, positive and bounded.

For the sake of convenience, we denote $\nabla\cdot d_1(x)\nabla={\cal L}_1$ and $\nabla\cdot d_2(x)\nabla={\cal L}_2$.

\section{Preliminaries on the equilibrium solutions of \qq{1.1} and \qq{1.2}}\lbl{s2}\setcounter{equation}{0}

Equilibrium problems of \qq{1.1} and \qq{1.2} are the following boundary value problems\vspace{-1mm}
 $$\begin{cases}
-{\cal L}_1 H_i^*=-\rho(x)H_i^*+\sigma_1(x)H_u(x)V_i^*,\, &x\in\oo,\\
-{\cal L}_2 V_u^*=-\sigma_2(x)V_u^*H_i^* +\beta(x)(V_u^*+V_i^*)-\mu(x)(V_u^*+V_i^*)V_u^*, &x\in\oo,\\
-{\cal L}_2 V_i^*=\sigma_2(x)V_u^*H_i^*-\mu(x)(V_u^*+V_i^*)V_i^*,\, &x\in\oo,\\
{\cal B}[H_i^*]={\cal B}[V_u^*]={\cal B}[V_i^*]=0,&x\in\partial\oo
	\end{cases}
	\eqno(2.1_{\cal B})\vspace{-1mm}$$
with ${\cal B}[W]=\partial_\nu W$ and ${\cal B}[W]=W$, respectively.
Let $\lm^{\cal B}(\beta)$ be the principal eigenvalue of
 \bess
	\begin{cases}
-{\cal L}_2\varphi-\beta(x)\varphi=\lm\varphi, &x\in\oo,\\
{\cal B}[\varphi]=0,&x\in\partial\oo.
	\end{cases}
	\eess
Assume that $(H_i^*, V_u^*, V_i^*)$ is a nonnegative solution of $(2.1_{\cal B})$ and set $V_{\cal B}=V_u^*+V_i^*$. Then $V_{\cal B}$ satisfies\vspace{-1mm}
 $$\begin{cases}
-{\cal L}_2 V_{\cal B}=\beta(x)V_{\cal B}-\mu(x)V_{\cal B}^2, &x\in\oo,\\
{\cal B}[V_{\cal B}]=0,&x\in\partial\oo.
	\end{cases}\eqno(2.2_{\cal B})\vspace{-1mm}$$
If $\lm^{\cal B}(\beta)\ge 0$, then $(2.2_{\cal B})$ has no positive solution. Therefore, $V_{\cal B}=0$, i.e., $V_u^*=V_i^*=0$, and then $H_i^*=0$ by the first equation of $(2.1_{\cal B})$. So $(0,0,0)$ is the only nonnegative solution of $(2.1_{\cal B})$. If $\lm^{\cal B}(\beta)<0$, then $(2.2_{\cal B})$ has a unique positive solution $V_{\cal B}$ and $V_{\cal B}$ is globally asymptotically stable.

Assume $\lm^{\cal B}(\beta)<0$ and let $V_{\cal B}$ be the unique positive solution of $(2.2_{\cal B})$. To investigate positive solutions of $(2.1_{\cal B})$ is equivalent to study positive solutions $(H_i^*, V_i^*)$ of\vspace{-1mm}
 $$\begin{cases}
-{\cal L}_1 H_i^*=-\rho(x)H_i^*+\sigma_1(x)H_u(x)V_i^*,\, &x\in\oo,\\
-{\cal L}_2 V_i^*=\sigma_2(x)(V_{\cal B} (x)-V_i^*)H_i^*-\mu(x)V_{\cal B} (x)V_i^*,\, &x\in\oo,\\
{\cal B}[H_i^*]={\cal B}[V_i^*]=0,&x\in\partial\oo
	\end{cases}\eqno(2.3_{\cal B})\vspace{-1mm}$$
satisfying $V_i^*<V_{\cal B}$. Clearly, $(0, 0)$ is the unique trivial nonnegative solution of $(2.3_{\cal B})$. The linearized eigenvalue problem of $(2.3_{\cal B})$ at $(0,0)$ is
 $$\begin{cases}
-{\cal L}_1\phi_1+\rho(x)\phi_1-\sigma_1(x)H_u(x)\phi_2=\lm\phi_1,\, &x\in\oo,\\
-{\cal L}_2\phi_2-\sigma_2(x)V_{\cal B} (x)\phi_1
+\mu(x)V_{\cal B} (x)\phi_2=\lm\phi_2,\, &x\in\oo,\\
{\cal B}[\phi_1]={\cal B}[\phi_2]=0,&x\in\partial\oo.
	\end{cases}\eqno(2.4_{\cal B})$$
By the results of \cite{Sw92}, the problem $(2.4_{\cal B})$ has a unique principal eigenvalue $\lm^{\cal B}(V_{\cal B})$ with positive eigenfunction $\phi=(\phi_1,\phi_2)^T$. Moreover, when ${\cal B}[W]=\partial_\nu W$ (the Neumann boundary condition),  $\lm^{\cal B}(V_{\cal B} )$ has the same sign with $\mathcal{R}_0-1$, where $\mathcal{R}_0$ is given in \cite{MWW18}.

\begin{theo}\lbl{th2.1} Assume $\lm^{\cal B}(\beta)<0$. If $(2.3_{\cal B})$ has a positive solution $(H_i^*, V_i^*)$, then $\lm^{\sst{\cal B}}(V_{\sst{\cal B}})<0$.
\end{theo}

\begin{proof} Define an operator
 \bess
 \mathscr{L}_{\sst{\cal B}}\psi=\left(\begin{array}{cc}-{\cal L}_1\psi_1+\rho(x)\psi_1,\\
 -{\cal L}_2\psi_2+\mu(x)V_{\sst{\cal B}}(x)\psi_2\end{array}\right),\;\;\; \psi_i\in C^2(\bar\oo), \;{\cal B}[\psi_i]_{\partial\oo}=0,\;i=1,2, \eess
and a matrix
 \bess
 A_{\sst{\cal B}}(x)=\left(\begin{array}{cc}
 \lm^{\sst{\cal B}}(V_{\sst{\cal B}}) & \sigma_1(x)H_u(x)\\ \sigma_2(x)V_{\sst{\cal B}}(x) & \lm^{\sst{\cal B}}(V_{\sst{\cal B}})\end{array}\right).
 \eess
Then $\mathscr{L}_{\sst{\cal B}}$ is reversible and $\mathscr{L}_{\sst{\cal B}}^{-1}$ is strongly positive and compact. On the contrary we assume $\lm^{\sst{\cal B}}(V_{\sst{\cal B}})\ge 0$. Then the operator $T_{\sst{\cal B}}:=\mathscr{L}_{\sst{\cal B}}^{-1}A_{\sst{\cal B}}(x)$ is also strongly positive and compact, and $r(T_{\sst{\cal B}})=1$ as $\phi_1, \phi_2>0$ and $\phi=(\phi_1, \phi_2)^T$ satisfies $\phi=\mathscr{L}_{\sst{\cal B}}^{-1}A_{\sst{\cal B}}(x)[\phi]$ by $(2.4_{\sst{\cal B}})$.

On the other hand, since $(H_i^*, V_i^*)$ is a positive solution of $(2.3_{\sst{\cal B}})$, we have
 \bess
 r(T_{\sst{\cal B}})\left(\begin{array}{cc}H_i^*\\[-0.5mm]
 V_i^*\end{array}\right)=\left(\begin{array}{cc}H_i^*\\[-0.5mm]
 V_i^*\end{array}\right)=T_{\sst{\cal B}}\left(\begin{array}{cc}H_i^*\\[-0.5mm]
 V_i^*\end{array}\right)-\mathscr{L}_{\sst{\cal B}}^{-1}\left(\begin{array}{cc}
 \lm^{\sst{\cal B}}(V_{\sst{\cal B}}) & 0\\[-0.5mm] \sigma_2(x)V_i^*\; & \; \lm^{\sst{\cal B}}(V_{\sst{\cal B}})\end{array}\right)\left(\begin{array}{cc}H_i^*\\[-0.5mm]
 V_i^*\end{array}\right),
 \eess
and
 \bess
 \mathscr{L}_{\sst{\cal B}}^{-1}\left(\begin{array}{cc}
 \lm^{\sst{\cal B}}(V_{\sst{\cal B}}) & 0\\[-0.5mm] \sigma_2(x)V_i^*\; & \; \lm^{\sst{\cal B}}(V_{\sst{\cal B}})\end{array}\right)\left(\begin{array}{cc}H_i^*\\[-0.5mm]
 V_i^*\end{array}\right)\ge\mathscr{L}_{\sst{\cal B}}^{-1}\left(\begin{array}{cc}
 0 \\[-0.5mm] \sigma_2(x)V_i^*H_i^*\end{array}\right)=:\left(\begin{array}{cc}
 0 \\[-0.5mm] \chi(x)\end{array}\right),
 \eess
where $\chi(x)>0$. This is impossible by the conclusion \cite[Theorem 3.2 (iv)]{Am76}. \end{proof}

\section{Dynamical properties of \qq{1.1}--main results in \cite{MWW18} and proofs}\lbl{s3}\setcounter{equation}{0}

In this section we focus on the model \qq{1.1} and write the boundary operator ${\cal B}$ as ${\cal N}$ to represent the Neumann boundary condition and $V_{\sst{\cal B}}=V_{\sst{\cal N}}$. Clearly, $\lm^{\sst{\cal N}}(\beta)<0$ since $\beta(x)>0$.

\begin{theo}\lbl{th3.1} Then the  problem $(2.3_{\sst{\cal N}})$ has a positive solution $(H_i^*, V_i^*)$ if and only if $\lm^{\sst{\cal N}}(V_{\sst{\cal N}})<0$. Moreover, $(H_i^*, V_i^*)$ is unique and satisfies $V_i^*<V_{\sst{\cal N}}$ when it exists. Therefore, $(2.1_{\sst{\cal N}})$  has a positive solution $(H_i^*, V_u^*, V_i^*)$ if and only if $\lm^{\sst{\cal N}}(V_{\sst{\cal N}})<0$, and $(H_i^*, V_u^*, V_i^*)$ is unique and takes the form $(H_i^*, V_{\sst{\cal N}}-V_i^*, V_i^*)$  when it exists.
\end{theo}

\begin{proof} By Theorem \ref{th2.1}, it suffices to prove that if $\lm^{\sst{\cal N}}(V_{\sst{\cal N}})<0$, then $(2.3_{\sst{\cal N}})$ has a unique positive solution $(H_i^*, V_i^*)$ and $V_i^*<V_{\sst{\cal N}}$. To this aim, we prove the following general conclusion.
\begin{enumerate}[leftmargin=6mm]
\item[$\bullet$] There is $0<\ep_0\ll 1$ such that, when $|\ep|\le\ep_0$, the problem\vspace{-1mm}
 $$\begin{cases}
-{\cal L}_1 H_i=-\rho(x)H_i+\sigma_1(x)H_u(x) V_i,\, &x\in\oo,\\
-{\cal L}_2 V_i=\sigma_2(x)(V_{\sst{\cal N}}(x)+\ep-V_i)^+ H_i-\mu(x)(V_{\sst{\cal N}}(x)-\ep)V_i,\, &x\in\oo,\\
 \partial_\nu H_i=\partial_\nu V_i=0,\;\;\;&x\in\partial\oo\vspace{-1mm}
	\end{cases}\eqno(3.1_\ep)$$
has a unique positive solution $(H_{i, \ep}^*, V_{i, \ep}^*)$, and $V_{i, \ep}^*<V_{\sst{\cal N}}+\ep$.
\end{enumerate}

{\it Existence of positive solution of $(3.1_\ep)$}. Let $\lm^{\sst{\cal N}}(V_{\sst{\cal N}}; \ep)$ be the principal eigenvalue of
 $$\begin{cases}
-{\cal L}_1\phi_1+\rho(x)\phi_1-\sigma_1(x)H_u(x)\phi_2
 =\lm\phi_1,\, &x\in\oo,\\
-{\cal L}_2\phi_2-\sigma_2(x)(V_{\sst{\cal N}}(x)+\ep)\phi_1
+\mu(x)(V_{\sst{\cal N}}(x)-\ep)\phi_2=\lm\phi_2,\, &x\in\oo,\\
 \partial_\nu\phi_1=\partial_\nu\phi_2=0,&x\in\partial\oo,
	\end{cases}\eqno(3.2_\ep)$$
(the existence and uniqueness of $\lm^{\sst{\cal N}}(V_{\sst{\cal N}}; \ep)$ is given in \cite{Sw92}). As $\lm^{\sst{\cal N}}(V_{\sst{\cal N}})<0$, there is $0<\ep_0\ll 1$ such that, when $|\ep|\le\ep_0$, we have $\lm^{\sst{\cal N}}(V_{\sst{\cal N}}; \ep)<0$, and $V_{\sst{\cal N}}(x)\pm\ep>0$, $\ep^2\mu(x)<\beta(x)V_{\sst{\cal N}}(x)$ in $\bar\oo$. Let $\bar H_i^*$ be the unique positive solution of the linear problem
 \bess\begin{cases}
-{\cal L}_1\bar H_i^*+\rho(x)\bar H_i^*=\sigma_1(x)H_u(x)(V_{\sst{\cal N}}(x)+\ep),\, &x\in\oo,\\
 \partial_\nu\bar H_i^*=0,&x\in\partial\oo.
	\end{cases} \eess
Then $(\bar H_i^*,V_{\sst{\cal N}}+\ep)$ is a strict upper solution of $(3.1_\ep)$.
Let $(\phi_1,\phi_2)$ be the positive eigenfunction corresponding to $\lm^{\sst{\cal N}}(V_{\sst{\cal N}}; \ep)$. It is easy to verify that $\delta(\phi_1, \phi_2)$ is a lower solution of $(3.1_\ep)$ and $\delta(\phi_1, \phi_2)\le(\bar H_i^*,V_{\sst{\cal N}}+\ep)$ provided $\delta>0$ is suitably small. By the upper and lower solutions method, $(3.1_\ep)$ has at least one positive solution $(H_{i, \ep}^*, V_{i, \ep}^*)$, and $V_{i, \ep}^*<V_{\sst{\cal N}}+\ep$, $H_{i, \ep}^*<\bar H_i^*$.\vskip 3pt

{\it Uniqueness of positive solutions of $(3.1_\ep)$}. Let $(\hat H_{i, \ep}^*, \hat V_{i, \ep}^*)$ be another positive solution of $(3.1_\ep)$. We can find $0<s<1$ such that $s(\hat H_{i,\ep}^*, \hat V_{i, \ep}^*)\le(H_{i, \ep}^*, V_{i, \ep}^*)$ in $\bar\oo$. Set
  \[\bar s=\sup\{0<s\le 1: s(\hat H_{i, \ep}^*, \hat V_{i, \ep}^*)\le(H_{i, \ep}^*, V_{i, \ep}^*) \;{\rm ~ in  ~ }\bar\oo\}.\]
Then $0<\bar s\le1$ and $\bar s(\hat H_{i, \ep}^*, \hat V_{i, \ep}^*)\le(H_{i, \ep}^*, V_{i, \ep}^*)$ in $\bar\oo$. We shall prove $\bar s=1$. If $\bar s<1$, then $W:=H_{i,\ep}^*-\bar s\hat H_{i,\ep}^*\ge 0, Z:=V_{i,\ep}^*-\bar s\hat V_{i,\ep}^*\ge 0$.
Since $\bar s\hat V_{i, \ep}^*\le V_{i, \ep}^*<V_{\sst{\cal N}}+\ep$ and $\hat V_{i, \ep}^*>\bar s\hat V_{i, \ep}^*$, we have
 \bess
 &(V_{\sst{\cal N}}+\ep-V_{i, \ep}^*)^+=V_{\sst{\cal N}}+\ep-V_{i, \ep}^*,\;\;\;
 (V_{\sst{\cal N}}+\ep-\hat V_{i, \ep}^*)^+<V_{\sst{\cal N}}+\ep-\bar s\hat V_{i, \ep}^*,&\\
& V_{\sst{\cal N}}(x)+\ep-V_{i, \ep}^*-(V_{\sst{\cal N}}(x)+\ep-\hat V_{i, \ep}^*)^+>\bar s\hat V_{i, \ep}^*-V_{i, \ep}^*=-Z.&
 \eess
After careful calculation, it derives that
\bess\begin{cases}
-{\cal L}_1 W=-\rho(x)W+\sigma_1(x)H_u(x)Z,\\
-{\cal L}_2 Z>-[\mu(x)(V_{\sst{\cal N}}(x)-\ep)+\sigma_2(x)\hat H_{i,\ep}^*]Z+\sigma_2(x)(V_{\sst{\cal N}}(x)+\ep-V_{i, \ep}^*)W.
	\end{cases}\eess
Notice $V_{\sst{\cal N}}(x)+\ep-V_{i, \ep}^*>0$ and $W, Z\ge 0$. It follows that $W,Z>0$ in $\bar\oo$ by the maximum principle. Then there exists $0<\tau<1-\bar s$ such that $(W, Z)\ge\tau(\hat H_{i, \ep}^*, \hat V_{i, \ep}^*)$, i.e., $(\bar s+\tau)(\hat H_{i, \ep}^*, \hat V_{i, \ep}^*)\le(H_{i, \ep}^*, V_{i, \ep}^*)$ in $\bar\oo$. This contradicts the definition of $\bar s$. Hence $\bar s=1$, i.e., $(\hat H_{i, \ep}^*, \hat V_{i, \ep}^*)\le(H_{i, \ep}^*, V_{i, \ep}^*)$ in $\bar\oo$. Certainly, $\hat V_{i, \ep}^*<V_{\sst{\cal N}}+\ep$, and $(V_{\sst{\cal N}}(x)+\ep-\hat V_{i, \ep}^*)^+=V_{\sst{\cal N}}(x)+\ep-\hat V_{i, \ep}^*$.

On the other hand, we can find $\zeta>1$ such that $\zeta(\hat H_{i,\ep}^*, \hat V_{i, \ep}^*)\ge(H_{i, \ep}^*, V_{i, \ep}^*)$ in $\oo$. Set
  \[\ud\zeta=\inf\{\zeta\ge1: \zeta(\hat H_{i,\ep}^*, \hat V_{i, \ep}^*)\ge(H_{i, \ep}^*, V_{i, \ep}^*)\;{\rm ~ in  ~ }\bar\oo\}.\]
Then $\ud\zeta$ is well defined, $\ud\zeta\ge1$ and $\ud\zeta(\hat H_{i,\ep}^*, \hat V_{i, \ep}^*)\ge(H_{i, \ep}^*, V_{i, \ep}^*)$ in $\oo$. If $\ud\zeta>1$, then $P:=\ud\zeta\hat H_{i,\ep}^*-H_{i, \ep}^*\ge 0$ and $Q:=\zeta\hat V_{i,\ep}^*-V_{i, \ep}^*\ge 0$, and $(P, Q)$ satisfies
 \bess\begin{cases}
-{\cal L}_1 P=-\rho(x)P+\sigma_1(x)H_u(x)Q,\\
-{\cal L}_2 Q>-[\mu(x)(V_{\sst{\cal N}}(x)-\ep)+\sigma_2(x) H_{i,\ep}^*]Q+\sigma_2(x)(V_{\sst{\cal N}}(x)+\ep-\hat V_{i, \ep}^*)P.
	\end{cases}\eess
As $V_{\sst{\cal N}}(x)+\ep-\hat V_{i, \ep}^*>0$ and $P, Q\ge 0$. Then $P,Q>0$ in $\bar\oo$, and there exists $0<r<\ud\zeta-1$ such that $(P, Q)\ge r(\hat H_{i,\ep}^*, \hat V_{i, \ep}^*)$, i.e., $(\ud\zeta-r)(\hat H_{i,\ep}^*, \hat V_{i, \ep}^*)\ge(H_{i, \ep}^*, V_{i, \ep}^*)$ in $\bar\oo$. This contradicts the definition of $\ud\zeta$. Hence $\ud\zeta=1$ and $(\hat H_{i,\ep}^*, \hat V_{i, \ep}^*)\ge(H_{i, \ep}^*, V_{i, \ep}^*)$. The uniqueness is proved.

Taking $\ep=0$ in the above, we obtain the desired conclusion.
\end{proof}\setcounter{equation}{2}

\begin{theo}\lbl{th3.2} Let $(H_i, V_u,V_i)$ be the unique positive solution of \qq{1.1}. Then the following hold.

{\rm(i)}\; If $\lm^{\sst{\cal N}}(V_{\sst{\cal N}})<0$, then
 \be
\lim_{t\to\yy}(H_i, V_u, V_i)=(H_i^*, V_{\sst{\cal N}}-V_i^*, V_i^*)=(H_i^*, V_u^*, V_i^*)\;\;\;
{\rm in}\;\; [C^2(\bar\oo)]^3.\vspace{-2mm}
 \lbl{3.3}\ee

{\rm(ii)}\; If $\lm^{\sst{\cal N}}(V_{\sst{\cal N}})\ge0$, then
 \be
\lim_{t\to\yy}(H_i, V_u, V_i)=(0, V_{\sst{\cal N}}, 0)\;\;\;
{\rm in}\;\; [C^2(\bar\oo)]^3.\vspace{-2mm}
 \lbl{3.4}\ee
  \end{theo}

\begin{proof} (i)\, Let $(H_i, V_u, V_i)$ be the unique solution of \qq{1.1}, and set $V=V_u+V_i$. Then we have
 \bess
	\begin{cases}
\partial_t V-{\cal L}_2 V=\beta(x)V-\mu(x)V^2, &x\in\oo,\;t>0,\\
\partial_\nu V=0,&x\in\partial\oo,\;t>0,\\
V(x,0)=V_u(x,0)+V_i(x,0)>0,&x\in\bar\oo,
	\end{cases}
	\eess
and $\dd\lim_{t\to\yy}V(x,t)=V_{\sst{\cal N}}(x)$ in $C^2(\bar\oo)$. For any given $0<\ep\le\ep_0$, there exists $T_\ep\gg 1$ such that
 \bess
0<V_{\sst{\cal N}}(x)-\ep\le V(x,t)\le V_{\sst{\cal N}}(x)+\ep,\;\;\;\forall\;x\in\bar\oo, \; t\ge T_\ep.
 \eess
Hence, $(H_i, V_i)$ satisfies
 \bess\begin{cases}
\partial_t H_i-{\cal L}_1 H_i=-\rho(x)H_i+\sigma_1(x) H_u(x)V_i,\, &x\in\oo,\; t>T_\ep,\\
\partial_t V_i-{\cal L}_2 V_i\leq\sigma_2(x)(V_{\sst{\cal N}}(x)+\ep-V_i)^+H_i-\mu(x)(V_{\sst{\cal N}}(x)-\ep)V_i,\, &x\in\oo,\; t>T_\ep,\\
\partial_\nu H_i=\partial_\nu V_i=0,\;\;\;&x\in\partial\oo,\; t>T_\ep.
 \end{cases}\eess

Let $(H_{i, \ep}^*, V_{i, \ep}^*)$ be the unique positive solution of $(3.1_\ep)$, and $(\phi_1,\phi_2)$ be the positive eigenfunction corresponding to $\lm^{\sst{\cal N}}(V_{\sst{\cal N}}; \ep)$. We can take constants $k\gg 1$ and $0<\delta\ll 1$ such that
 \bess
 k(H_{i,\ep}^*(x), V_{i, \ep}^*(x))\ge(H_i(x, T_\ep), V_i(x, T_\ep)),\;\;\; \delta(\phi_1(x),\phi_2(x))\le(H_i(x, T_\ep), V_i(x, T_\ep)),\;\;\;x\in\bar\oo.
 \eess
Then $k(H_{i,\ep}^*, V_{i, \ep}^*)$ and $\delta(\phi_1,\phi_2)$ are the upper and lower solutions of $(3.1_\ep)$, respectively.
Let $(H_{i,\ep}, V_{i,\ep})$ be the unique positive solution of
 \bes\begin{cases}
\partial_t H_i-{\cal L}_1 H_i=-\rho(x) H_i+\sigma_1(x)H_u(x)V_i,\, &x\in\oo,\; t>0,\\
\partial_t V_i-{\cal L}_2 V_i=\sigma_2(x)(V_{\sst{\cal N}}(x)+\ep-V_i)^+ H_i-\mu(x)(V_{\sst{\cal N}}(x)-\ep)V_i,\, &x\in\oo,\; t>0,\\
\partial_\nu H_i=\partial_\nu V_i=0,\;\;\;&x\in\partial\oo,\; t>0,\\
(H_i(x, 0),\,V_i(x,0))=k(H_{i, \ep}^*,\,V_{i, \ep}^*),&x\in\bar\oo.
	\end{cases}\lbl{3.5}\ees
Then $(H_{i,\ep}(x,t), V_{i,\ep}(x,t))$ is decreasing in $t$, and
\bess
 \delta(\phi_1(x),\phi_2(x))\le (H_{i,\ep}(x,t), V_{i,\ep}(x,t)),\;\;(H_i(x,t+T_\ep), V_i(x,t+T_\ep))\le (H_{i,\ep}(x,t), V_{i,\ep}(x,t))
 \eess
by the comparison principle. Moreover, $\dd\lim_{t\to\yy}(H_{i,\ep}(x,t), V_{i,\ep}(x,t))=(H_{i, \ep}^*, V_{i, \ep}^*)$ in $[C^2(\bar\oo)]^2$
by the uniform estimate and compact arguments (cf. \cite[Theorems 2.11, 3.14]{Wpara}) and the uniqueness of positive solutions of $(3.1_\ep)$.  So $\dd\limsup_{t\to\yy}(H_i(x,t), V_i(x,t))\le (H_{i, \ep}^*, V_{i, \ep}^*)$ and, by letting $\ep\to 0$,
 \bes
\limsup_{t\to\yy}(H_i(x,t), V_i(x,t))\le (H_i^*, V_i^*)\;\;\;
{\rm uniformly\; in}\;\; \bar\oo.
  \lbl{3.6}\ees

On the other hand, as $V_i^*(x)<V_{\sst{\cal N}}(x)$ in $\bar\oo$, there exists $0<r_0\le\ep_0$ such that $V_i^*(x)<V_{\sst{\cal N}}(x)-2r$ in $\bar\oo$ for all $0<r<r_0$. For such a $r$, using \qq{3.6} and $\dd\lim_{t\to\yy}V(x,t)=V_{\sst{\cal N}}(x)$ in $C^2(\bar\oo)$, we can can find $T_r\gg 1$ such that
 \bess
V_i(x,t)\le V_i^*(x)+r<V_{\sst{\cal N}}(x)-r,\;\; 0<V_{\sst{\cal N}}(x)-r\le V(x,t)\le V_{\sst{\cal N}}(x)+r,\;\;\;\forall\;x\in\bar\oo, \; t\ge T_r.
 \eess
Consequently, $(H_i, V_i)$ satisfies
 \bess\begin{cases}
\partial_t H_i-{\cal L}_1 H_i=-\rho(x)H_i+\sigma_1(x) H_u(x)V_i,\, &x\in\oo,\; t>T_r,\\
\partial_t V_i-{\cal L}_2 V_i\geq\sigma_2(x)(V_{\sst{\cal N}}(x)-r-V_i)H_i-\mu(x)(V_{\sst{\cal N}}(x)+r)V_i\\
\hspace{20mm}=\sigma_2(x)(V_{\sst{\cal N}}(x)-r-V_i)^+H_i-\mu(x)(V_{\sst{\cal N}}(x)+r)V_i,\, &x\in\oo,\; t>T_r,\\
\partial_\nu H_i=\partial_\nu V_i=0,\;\;\;&x\in\partial\oo,\; t>T_r,\\
H_i(x, T_r)>0,\; 0<V_i(x, T_r)<V_{\sst{\cal N}}(x)-r,&x\in\bar\oo.
 \end{cases}\eess

Let $(\phi_1,\phi_2)$ be the positive eigenfunction corresponding to $\lm^{\sst{\cal N}}(V_{\sst{\cal N}}; -r)$. Similar to the above, $\delta(\phi_1,\phi_2)$ is a lower solution of $(3.1_{-r})$ and $\delta(\phi_1(x),\phi_2(x))\le (H_i(x, T_r), V_i(x, T_r))$ in $\bar\oo$ provided $0<\delta\ll 1$. Let $(H_{i,-r}, V_{i,-r})$ be the unique positive solution of \qq{3.5} with $\ep=-r$ and $(H_i(x, 0),\,V_i(x,0))=\delta(\phi_1(x),\phi_2(x))$. Then
 \[(H_i(x,t+T_r), V_i(x,t+T_r))\ge (H_{i,-r}(x,t),V_{i,-r}(x,t)),\;\; x\in\oo,\; t>0, \]
and $(H_{i,-r}, V_{i,-r})$ is increasing in $t$ by the comparison principle. Similar to the above,
 $$\dd\lim_{t\to\yy}(H_{i,-r}(x,t), V_{i,-r}(x,t))=(H_{i,-r}^*, V_{i, -r}^*)\;\;\;{\rm in }\;\; [C^2(\bar\oo)]^2,$$
where $(H_{i,-r}^*, V_{i, -r}^*)$ is the unique positive solution of $(3.1_{-r})$. Thus, $\dd\liminf_{t\to\yy}(H_i, V_i)\ge (H_{i,-r}^*, V_{i, -r}^*)$, and then $\dd\liminf_{t\to\yy}(H_i, V_i)\ge (H_i^*, V_i^*)$ by letting $r\to 0$. This combined with \qq{3.6} yields $\dd\lim_{t\to\yy}(H_i, V_i)=(H_i^*, V_i^*)$ uniformly in $\bar\oo$. Using the fact that $\dd\lim_{t\to\yy}(V_u+V_i)=V_{\sst{\cal N}}$ in $C^2(\bar\oo)$ and the uniform estimate (cf. \cite[Theorems 2.11, 3.14]{Wpara}), we see that \qq{3.3} holds.

(ii) Let $\lm^{\sst{\cal N}}(V_{\sst{\cal N}}; \ep)$ be the principal eigenvalue of $(3.2_\ep)$. If $\lm^{\sst{\cal N}}(V_{\sst{\cal N}})>0$, then $\lm^{\sst{\cal N}}(V_{\sst{\cal N}}; \ep)>0$ when $0<\ep\ll 1$. So $(3.1_\ep)$ has no positive solution. Let $(H_{i,\ep}, V_{i,\ep})$ be the unique positive solution of \qq{3.5} with initial data $(C, K)$, where $C$ and $K$ are suitably large positive constants, for example,
  \bess
 K>\max_{\bar\oo}\max\{V_{\sst{\cal N}}(x)+\ep,\,V_i(x,T_\ep)\}, \;\;C>\max_{\bar\oo}\max\{K\sigma_1(x)H_u(x)/\rho(x),\, H_i(x,T_\ep)\}.\eess
Then $(C, K)$ is an upper solution of $(3.1_\ep)$, and $(H_{i,\ep}, V_{i,\ep})$ is decreasing in $t$. Similar to the above,  $\dd\lim_{t\to\yy}(H_{i,\ep}, V_{i,\ep})=(0,0)$ since $(3.1_\ep)$ has no positive solution. Then, by the comparison principle, $\dd\lim_{t\to\yy}(H_i, V_i)=(0,0)$. This together with $\dd\lim_{t\to\yy}(V_u+V_i)=V_{\sst{\cal N}}$ gives the limit \qq{3.4}.\vskip 4pt

If $\lm^{\sst{\cal N}}(V_{\sst{\cal N}})=0$, then $\lm^{\sst{\cal N}}(V_{\sst{\cal N}}; \ep)<0$ and $(3.1_{\ep})$ has a unique positive solution $(H_{i,\ep}^*, V_{i,\ep}^*)$ for any $\ep>0$. Obviously, $\dd\lim_{\ep\to 0}(H_{i,\ep}^*, V_{i,\ep}^*)=(0, 0)$ in $[C^2(\bar\oo)]^2$ since $(3.1_0)$, i.e., $(2.3_{\sst{\cal N}})$ has no positive solution in the present case. Similar to the above, $\dd\lim_{t\to\yy}(H_i, V_i)=(0,0)$ and \qq{3.4} holds. The proof is complete. \end{proof}

\section{Dynamical properties of \qq{1.2}}\setcounter{equation}{0}

In the present situation we write the boundary operator ${\cal B}$ as ${\cal D}$ to represent the Dirichlet boundary condition and $V_{\sst{\cal B}}=V_{\sst{\cal D}}$. The results and arguments are similar to that in Section \ref{s3}.

\begin{theo}\lbl{th4.1} Assume $\lm^{\sst{\cal D}}(\beta)<0$. Then the  problem $(2.3_{\sst{\cal D}})$ has a positive solution $(H_i^*, V_i^*)$ if and only if $\lm^{\sst{\cal D}}(V_{\sst{\cal D}})<0$. Moreover, $(H_i^*, V_i^*)$ is unique and satisfies $V_i^*<V_{\sst{\cal D}}$ when it exists. Therefore, $(2.1_{\sst{\cal D}})$ has a positive solution $(H_i^*, V_u^*, V_i^*)$ if and only if $\lm^{\sst{\cal D}}(V_{\sst{\cal D}})<0$, and $(H_i^*, V_u^*, V_i^*)$ is unique and takes the form $(H_i^*, V_{\sst{\cal D}}-V_i^*, V_i^*)$  when it exists.
\end{theo}

\begin{proof} By Theorem \ref{th2.1}, it suffices to prove that if $\lm^{\sst{\cal D}}(V_{\sst{\cal D}})<0$, then $(2.3_{\sst{\cal D}})$ has a unique positive solution $(H_i^*, V_i^*)$ and $V_i^*<V_{\sst{\cal D}}$. To this aim, we prove the following general conclusion.\vspace{-2mm}
\begin{enumerate}[leftmargin=6mm]
\item[$\bullet$] Let $\varphi(x)$, with $\max_{\bar\oo}\|\varphi\|=1$, be the positive eigenfunction corresponding to $\lm^{\sst{\cal D}}(\beta)$. Then there is $0<\ep_0\ll 1$ such that, when $|\ep|\le\ep_0$, the problem\vspace{-2mm}
 \hspace{-3mm}$$\begin{cases}
-{\cal L}_1 H_i^*=-\rho(x)H_i^*+\sigma_1(x)H_u(x) V_i^*,\, &x\in\oo,\\
-{\cal L}_2 V_i^*=\sigma_2(x)(V_{\sst{\cal D}}(x)+\ep\varphi(x)-V_i^*)^+ H_i^*-\mu(x)(V_{\sst{\cal D}}(x)-\ep\varphi(x))V_i^*,\, &x\in\oo,\\
  H_i^*= V_i^*=0,\;\;\;&x\in\partial\oo\vspace{-1mm}
	\end{cases}\eqno(4.1_\ep)$$
has a unique positive solution $(H_{i, \ep}^*, V_{i, \ep}^*)$, and $V_{i, \ep}^*(x)<V_{\sst{\cal D}}(x)+\ep\varphi(x)$. \end{enumerate}

{\it Existence of positive solution of $(4.1_\ep)$}. Let $\lm^{\sst{\cal D}}(V_{\sst{\cal D}}; \ep)$ be the principal eigenvalue of
 $$\begin{cases}
-{\cal L}_1\phi_1+\rho(x)\phi_1-\sigma_1(x)H_u(x)\phi_2
 =\lm\phi_1,\, &x\in\oo,\\
-{\cal L}_2\phi_2-\sigma_2(x)(V_{\sst{\cal D}}(x)+\ep\varphi(x))\phi_1
+\mu(x)(V_{\sst{\cal D}}(x)-\ep\varphi(x))\phi_2=\lm\phi_2,\, &x\in\oo,\\
 \phi_1=\phi_2=0,&x\in\partial\oo.
	\end{cases}$$
As $\lm^{\sst{\cal D}}(V_{\sst{\cal D}})<0$, there is $0<\ep_0\ll 1$ such that, when $|\ep|\le\ep_0$, we have $\lm^{\sst{\cal D}}(V_{\sst{\cal D}}; \ep)<0$, and (cf. \cite[Lemma 2.1]{WPE24})
 $$V_{\sst{\cal D}}(x)\pm\ep\varphi(x)>0, \;\;\;(\ep\varphi(x))^2\mu(x)
 <\beta(x)V_{\sst{\cal D}}(x)+\ep\varphi(x)(\lm^{\sst{\cal D}}(V_{\sst{\cal D}})+\beta(x))\;\;\;
 {\rm in }\;\; \oo.$$
Let $\bar H_i^*$ be the unique positive solution of the linear problem
 \bess\begin{cases}
-{\cal L}_1\bar H_i^*+\rho(x)\bar H_i^*=\sigma_1(x)H_u(x)(V_{\sst{\cal D}}(x)+\ep\varphi(x)),\, &x\in\oo,\\
 \bar H_i^*=0,&x\in\partial\oo.
	\end{cases} \eess
Then $(\bar H_i^*,V_{\sst{\cal D}}+\ep\varphi(x))$ is a strict upper solution of $(4.1_\ep)$. Let $(\phi_1,\phi_2)$ be the positive eigenfunction corresponding to $\lm^{\sst{\cal D}}(V_{\sst{\cal D}}; \ep)$. It is easy to verify that $\delta(\phi_1, \phi_2)$ is a lower solution of $(4.1_\ep)$ and $\delta(\phi_1, \phi_2)\le(\bar H_i^*,V_{\sst{\cal D}}+\ep\varphi(x))$ provided $\delta>0$ is suitably small. By the upper and lower solutions method, $(4.1_\ep)$ has at least one positive solution $(H_{i, \ep}^*, V_{i, \ep}^*)$, and $V_{i, \ep}^*<V_{\sst{\cal D}}+\ep\varphi(x)$, $H_{i, \ep}^*<\bar H_i^*$.\vskip 2pt

{\it Uniqueness of positive solutions of $(4.1_\ep)$}. Let $(\hat H_{i, \ep}^*, \hat V_{i, \ep}^*)$ be another positive solution of $(4.1_\ep)$. We can find $0<s<1$ such that $s(\hat H_{i,\ep}^*, \hat V_{i, \ep}^*)\le(H_{i, \ep}^*, V_{i, \ep}^*)$ in $\oo$. Set
  \[\bar s=\sup\{0<s\le 1: s(\hat H_{i, \ep}^*, \hat V_{i, \ep}^*)\le(H_{i, \ep}^*, V_{i, \ep}^*) \;{\rm ~ in  ~ }\oo\}.\]
Then $\bar s$ is well defined, $0<\bar s\le1$ and $\bar s(\hat H_{i, \ep}^*, \hat V_{i, \ep}^*)\le(H_{i, \ep}^*, V_{i, \ep}^*)$ in $\oo$. We shall prove $\bar s=1$. If $\bar s<1$, then $W:=H_{i,\ep}^*-\bar s\hat H_{i,\ep}^*\ge 0, Z:=V_{i,\ep}^*-\bar s\hat V_{i,\ep}^*\ge 0$.
Since $\bar s\hat V_{i, \ep}^*\le V_{i, \ep}^*<V_{\sst{\cal D}}+\ep\varphi(x)$ and $\bar s<1$, we have
 \bess
 &(V_{\sst{\cal D}}+\ep\varphi(x)-V_{i, \ep}^*)^+=V_{\sst{\cal D}}+\ep\varphi(x)-V_{i, \ep}^*,\;\;\;
 (V_{\sst{\cal D}}+\ep\varphi(x)-\hat V_{i, \ep}^*)^+<V_{\sst{\cal D}}+\ep\varphi(x)-\bar s\hat V_{i, \ep}^*,&\\
& V_{\sst{\cal D}}(x)+\ep\varphi(x)-V_{i, \ep}^*-(V_{\sst{\cal D}}(x)+\ep\varphi(x)-\hat V_{i, \ep}^*)^+>\bar s\hat V_{i, \ep}^*-V_{i, \ep}^*=-Z.&
 \eess
By the careful calculations,
\bess\begin{cases}
-{\cal L}_1 W=-\rho(x)W+\sigma_1(x)H_u(x)Z,\\
-{\cal L}_2 Z>-[\mu(x)(V_{\sst{\cal D}}(x)-\ep\varphi(x))+\sigma_2(x)\hat H_{i,\ep}^*]Z+\sigma_2(x)(V_{\sst{\cal D}}(x)+\ep\varphi(x)-V_{i, \ep}^*)W.
	\end{cases}\eess
Notice $V_{\sst{\cal D}}(x)+\ep\varphi(x)-V_{i, \ep}^*>0$ and $W, Z\ge 0$. It follows that $W,Z>0$ in $\oo$ by the maximum principle. Then there exists $0<r<1-\bar s$ such that $(W, Z)\ge r(\hat H_{i, \ep}^*, \hat V_{i, \ep}^*)$, i.e., $(\bar s+r)(\hat H_{i, \ep}^*, \hat V_{i, \ep}^*)\le(H_{i, \ep}^*, V_{i, \ep}^*)$ in $\oo$. This contradicts the definition of $\bar s$. Hence $\bar s=1$, i.e., $(\hat H_{i, \ep}^*, \hat V_{i, \ep}^*)\le(H_{i, \ep}^*, V_{i, \ep}^*)$ in $\oo$. Certainly, $\hat V_{i, \ep}^*<V_{\sst{\cal D}}+\ep\varphi(x)$, and $(V_{\sst{\cal D}}(x)+\ep\varphi(x)-\hat V_{i, \ep}^*)^+=V_{\sst{\cal D}}(x)+\ep\varphi(x)-\hat V_{i, \ep}^*$.

Same as the proof of Theorem \ref{th3.1} we can prove $(\hat H_{i,\ep}^*, \hat V_{i, \ep}^*)\ge(H_{i, \ep}^*, V_{i, \ep}^*)$.

Take $\ep=0$ in the above, we obtain the desired conclusion.
\end{proof}\setcounter{equation}{1}

\begin{theo} Let $(H_i, V_u,V_i)$ be the unique positive solution of \qq{1.2}. Then the following holds.

{\rm(i)}\; If $\lm^{\sst{\cal D}}(\beta)<0$ and $\lm^{\sst{\cal D}}(V_{\sst{\cal D}})<0$, then $\dd\lim_{t\to\yy}(H_i, V_u, V_i)=(H_i^*, V_{\sst{\cal D}}-V_i^*, V_i^*)=(H_i^*, V_u^*, V_i^*)$ in $[C^2(\bar\oo)]^3$.

{\rm(ii)}\; If  $\lm^{\sst{\cal D}}(\beta)<0$ and $\lm^{\sst{\cal D}}(V_{\sst{\cal D}})\ge0$, then
$\dd\lim_{t\to\yy}(H_i, V_u, V_i)=(0, V_{\sst{\cal D}}, 0)$ in $[C^2(\bar\oo)]^3$.

{\rm(iii)}\; If $\lm^{\sst{\cal D}}(\beta)\ge 0$, then $\dd\lim_{t\to\yy}(H_i, V_u, V_i)=(0, 0, 0)$ in $[C^2(\bar\oo)]^3$.
   \end{theo}

\begin{proof} (i)\, Let $(H_i, V_u, V_i)$ be the unique solution of \qq{1.2}, and set $V=V_u+V_i$. Then we have
 \bess
	\begin{cases}
\partial_t V-{\cal L}_2 V=\beta(x)V-\mu(x)V^2, &x\in\oo,\;t>0,\\
 V=0,&x\in\partial\oo,\;t>0,\\
V(x,0)=V_u(x,0)+V_i(x,0)>0,&x\in\bar\oo,
	\end{cases}
	\eess
and $\dd\lim_{t\to\yy}V(x,t)=V_{\sst{\cal D}}(x)$ in $C^2(\bar\oo)$. Let $\varphi(x)$, with $\max_{\bar\oo}\|\varphi\|=1$, be the positive eigenfunction corresponding to $\lm^{\sst{\cal D}}(\beta)$. We first prove that for any given $0<\ep\le\ep_0$, there exists $T_\ep\gg 1$ such that
 \bes
0<V_{\sst{\cal D}}(x)-\ep\varphi(x)<V(x,t)<V_{\sst{\cal D}}(x)+\ep\varphi(x),\;\;\;
\forall\;x\in\oo,\; t\ge T_\ep.
 \label{4.2}\ees
The idea of the proof of \qq{4.2} comes from that of \cite[Lemma 2.1]{WPE24}. In fact, set $W(x,t)=V(x,t)-V_{\sst{\cal D}}(x)$ and $Z(x)=\ep\varphi(x)$. Then $W(x,t)\to 0$ in $C^1(\bar\oo)$ as $t\to\yy$. Thanks to $Z\in C^1(\bar\oo)$ and  $\partial_\nu Z\big|_{\partial\oo}<0$, there exist $\tau>0$ and $T_1\gg 1$ such that $\partial_\nu(W(x,t)-Z(x))\big|_{\partial\oo}\ge\tau$ for all $t\ge T_1$. This together with $(W(x,t)-Z(x))\big|_{\partial\oo}=0$  asserts that there exists $\Omega_0\Subset\Omega$ such that $W(x,t)-Z(x)<0$ in $\Omega\setminus\Omega_0$. As $Z>0$ in $\bar\Omega_0$ and $W(x,t)\to 0$ in $C(\bar\Omega_0)$ as $t\to\yy$, we can find
$T_2\gg 1$ such that $W(x,t)-Z(x)<0$ in $\bar\Omega_0$. Thus, $W(x,t)-Z(x)<0$, i.e., $V(x,t)<V_{\sst{\cal D}}(x)+\ep\varphi(x)$ in $\oo$ for $t\ge\max\{T_1, T_2\}$. Similarly, we can prove $V_{\sst{\cal D}}(x)-\ep\varphi(x)<V(x,t)$ in $\oo$ when $t$ is large.

Making use of \qq{4.2}, we see that $(H_i, V_i)$ satisfies
 \bess\begin{cases}
\partial_t H_i-{\cal L}_1 H_i=-\rho(x)H_i+\sigma_1(x) H_u(x)V_i,\, &x\in\oo,\; t>T_\ep,\\
\partial_t V_i-{\cal L}_2 V_i\leq\sigma_2(x)(V_{\sst{\cal D}}(x)+\ep\varphi(x)-V_i)^+H_i-\mu(x)(V_{\sst{\cal D}}(x)-\ep\varphi(x))V_i,\, &x\in\oo,\; t>T_\ep,\\
 H_i= V_i=0,\;\;\;&x\in\partial\oo,\; t>T_\ep.
 \end{cases}\eess

Let $(H_{i, \ep}^*, V_{i, \ep}^*)$ be the unique positive solution of $(4.1_\ep)$, and $(\phi_1,\phi_2)$ be the positive eigenfunction corresponding to $\lm^{\sst{\cal D}}(V_{\sst{\cal D}}; \ep)$. According to \cite[Lemma 2.1]{WPE24}, we can take constants $k\gg 1$ and $0<\delta\ll 1$ such that
 \bess
 k(H_{i,\ep}^*(x), V_{i, \ep}^*(x))\ge(H_i(x, T_\ep), V_i(x, T_\ep)),\;\;\; \delta(\phi_1(x),\phi_2(x))\le(H_i(x, T_\ep), V_i(x, T_\ep)).
 \eess
Then $k(H_{i,\ep}^*, V_{i, \ep}^*)$ and $\delta(\phi_1,\phi_2)$ are the upper and lower solutions of $(4.1_\ep)$, respectively.
Let $(H_{i,\ep}, V_{i,\ep})$ be the unique positive solution of
 \bess\begin{cases}
\partial_t H_i-{\cal L}_1 H_i=-\rho(x) H_i+\sigma_1(x)H_u(x)V_i,\, &x\in\oo,\; t>0,\\
\partial_t V_i-{\cal L}_2 V_i=\sigma_2(x)(V_{\sst{\cal D}}(x)+\ep-V_i)^+ H_i-\mu(x)(V_{\sst{\cal D}}(x)-\ep)V_i,\, &x\in\oo,\; t>0,\\
 H_i= V_i=0,\;\;\;&x\in\partial\oo,\; t>0,\\
(H_i(x, 0),\,V_i(x,0))=k(H_{i, \ep}^*,\,V_{i, \ep}^*),&x\in\bar\oo.
	\end{cases}\eess
Then $(H_{i,\ep}(x,t), V_{i,\ep}(x,t))$ is decreasing in $t$, and
\bess
 \delta(\phi_1(x),\phi_2(x))\le (H_{i,\ep}(x,t), V_{i,\ep}(x,t)),\;\;(H_i(x,t+T_\ep), V_i(x,t+T_\ep))\le (H_{i,\ep}(x,t), V_{i,\ep}(x,t))
 \eess
by the comparison principle. Therefore, $\dd\lim_{t\to\yy}(H_{i,\ep}(x,t), V_{i,\ep}(x,t))=(H_{i, \ep}^*, V_{i, \ep}^*)$ in $[C^2(\bar\oo)]^2$
by the uniform estimate and compact arguments (cf. \cite[Theorems 2.11, 3.14]{Wpara}) and the uniqueness of positive solutions of $(4.1_\ep)$.  So $\dd\limsup_{t\to\yy}(H_i(x,t), V_i(x,t))\le (H_{i, \ep}^*, V_{i, \ep}^*)$ and, by letting $\ep\to 0$,
 \bess
\limsup_{t\to\yy}(H_i(x,t), V_i(x,t))\le (H_i^*, V_i^*)\;\;\;
{\rm uniformly\; in}\;\; \bar\oo.
  \eess

On the other hand, let $W(x)=V_{\sst{\cal D}}(x)-V_i^*(x)$. Then $W(x)>0$ as  $V_i^*(x)<V_{\sst{\cal D}}(x)$ in $\oo$. Through the direct calculations we have
 \bess
 -{\cal L}_2 W+\big[\sigma_2(x)H_i^*+\mu(x)V_{\sst{\cal D}}(x)\big]W=\beta(x)V_{\sst{\cal D}}>0.
 \eess
By the Hopf boundary lemma, $\partial_\nu W\big|_{\partial\oo}<0$.
Making use of \cite[Lemma 2.1]{WPE24}, we can find a $0<r_0\le\ep_0$ such that $W>2r\vp$, i.e.,
  \bess
 V_i^*(x)<V_{\sst{\cal D}}(x)-2r\vp(x)\;\;\;{\rm in }\;\; \oo\;\;\forall\; 0<r<r_0.\eess
For such $r$, as $\dd\lim_{\ep\to 0} V_{i, \ep}^*=V_i^*$ in $C^2(\bar\oo)$ (the uniform estimate and compact arguments), similar to the proof of \qq{4.2}, we can find $0<\bar\ep<\ep_0$ such that
   \bess
V_{i,\bar\ep}^*(x)<V_i^*(x)+\frac r2\vp(x)\;\;\;{\rm in }\;\; \oo.\eess
Owing to $\dd\lim_{t\to\yy}V_{i,\bar\ep}(x,t)=V_{i,\bar\ep}^*$ in $C^2(\bar\oo)$, by \qq{4.2}, there exists $\bar T_{\bar\ep}\gg 1$ such that
\bess
 V_{i,\bar\ep}(x,t)<V_{i,\bar\ep}^*(x)+\frac r2\vp(x)<V_i^*(x)+ r\vp(x)<V_{\sst{\cal D}}(x)-r\vp(x)\;\;\;{\rm in }\;\; \oo,\;\;\forall\; t\ge\bar T_{\bar\ep}. \eess
Therefore, when $t\ge T_{\bar\ep}+\bar T_{\bar\ep}$, we have $V_i(x,t)\le V_{i,\bar\ep}(x,t-T_{\bar\ep})<V_{\sst{\cal D}}(x)-r\vp(x)$ in $\oo$. This combined with \qq{4.2} indicates that there exists $T_r\gg 1$ such that
 \bess
 V_i(x,t)<V_{\sst{\cal D}}(x)-r\vp(x),\;\;0<V_{\sst{\cal D}}(x)-r\vp(x)\le V(x,t)\le V_{\sst{\cal D}}(x)+r\vp(x) \eess
for all $t\ge T_r$ and $x\in\oo$. Consequently, $(H_i, V_i)$ satisfies
 \bess\begin{cases}
\partial_t H_i-{\cal L}_1 H_i=-\rho(x)H_i+\sigma_1(x) H_u(x)V_i,\, &x\in\oo,\; t>T_r,\\
\partial_t V_i-{\cal L}_2 V_i\geq\sigma_2(x)(V_{\sst{\cal D}}(x)-r\vp-V_i)H_i-\mu(x)(V_{\sst{\cal D}}(x)+r\vp)V_i\\
\hspace{20mm}=\sigma_2(x)(V_{\sst{\cal D}}(x)-r\vp-V_i)^+H_i
-\mu(x)(V_{\sst{\cal D}}(x)+r\vp)V_i,\, &x\in\oo,\; t>T_r,\\
 H_i= V_i=0,\;\;\;&x\in\partial\oo,\; t>T_r,\\
H_i(x, T_r)>0,\; 0<V_i(x, T_r)<V_{\sst{\cal D}}(x)-r\vp,&x\in\bar\oo.
 \end{cases}\eess

The remaining proof is the same as that of Theorem \ref{th3.2}(i), and the details are omitted.

(ii) The proof is the same as that of Theorem \ref{th3.2}(ii), and the details are also omitted.

(iii) In such case, $(2.3_{\cal D})$ has no positive solution and $\lim_{t\to\yy}(V_u+V_i)=0$. Therefore, $\lim_{t\to\yy}H_i=0$ by the first equation of \qq{1.2}. The proof is complete.
\end{proof}

Define boundary operators ${\cal B}_j[W]=a_j(x)\partial_\nu W+b_j(x)W$, where either $a_j(x)\equiv 0, b_j(x)\equiv 1$ or $a_j(x)\equiv 1, b_j(x)\ge 0$ with $b_j\in C^{1+\alpha}(\partial\oo)$, $j=1,2$. It should be mentioned that when the boundary conditions in \qq{1.1} and \qq{1.2} are replaced by ${\cal B}_1[H_i]=0$, ${\cal B}_2[V_u]={\cal B}_2[V_i]=0$, then the conclusions of this paper are still true and the proofs are similar.

\vskip 4pt \noindent {\bf Acknowledgment:} The author greatly appreciates the suggestions provided by  Dr. Lei Li and Dr. Wenjie Ni in the proof of  \qq{4.2}.


\begin{thebibliography}{10}\setlength{\itemsep}
{0mm}\linespread{1.2}\selectfont

\bibitem{PAHO1}PAHO, Chikungunya, http://www.paho.org/hq/index. php?option=comtopics\&view=article \&id=343\&Itemid=40931\&lang=en.

\bibitem{PAHO2}PAHO, Zika virus infection, http://www.paho.org/hq/index.php?option=comcontent\&view= article\&id=11585\&Itemid=41688\&lang=en.

\bibitem{WHO1}WHO, Dengue and severe dengue, https://www.who.int/news-room/fact-sheets/detail/dengue -and-severe-dengue.

\bibitem{WHO2} WHO, EPI-WIN digest 7-Zika virus disease, https://www.who.int/publications/m/item/epi-win-digest-7-zika-virus-disease.

\bibitem{FMW} W.E. Fitzgibbon, J.J. Morgan and G.F. Webb, An outbreak vector-host epidemic model with  spatial structure: the 2015-2016 Zika outbreak in Rio De Janeiro. Theor. Biol. Med. Modell. 14, 2-17 (2017).

\bibitem{MWW18}
P. Magal, G.F. Webb and Y.X. Wu, On a vector-host epidemic model with spatial structure. Nonlinearity 31,  5589-5614 (2018).

\bibitem{MWW19}
P. Magal, G.F. Webb and Y.X. Wu, On the basic reproduction number of reaction-diffusion epidemic models. SIAM J. Appl. Math. 79, 284-304 (2019).

\bibitem{Sw92} G. Sweers, Strong positivity in $C(\overline\Omega)$ for elliptic systems. Math. Z. 209, 251-271 (1992).

\bibitem{Am76} H. Amann, Fixed point equations and nonlinear eigenvalue problems in ordered Banach spaces. SIAM Rev. 18, 620-709 (1976).

\vspace{-1.5mm}\bibitem{Wpara} \newblock M. X. Wang,
\newblock \emph{Nonlinear Second Order Parabolic Equations}.
\newblock Boca Raton: CRC Press, 2021.

\vspace{-1.5mm}\bibitem{WPE24} M.X. Wang and P.Y.H. Pang, Nonlinear Second Order Elliptic Equations. Heidelberg: Springer, 2024.


\end{thebibliography}
\end{document}